\documentclass[preprint,12pt,3p]{elsarticle}
\usepackage{verbatim}
\usepackage{amsmath,amsthm,amsfonts,amssymb}
\usepackage{algorithm}
\usepackage{algpseudocode}
\usepackage{cases}
\usepackage[hidelinks]{hyperref}
\usepackage{graphics}
\usepackage{booktabs}
\usepackage{geometry}
\usepackage[utf8]{inputenc}
\usepackage{newunicodechar}
\usepackage{graphicx}
\usepackage{caption}
\usepackage{tikz}
\usepackage{subcaption}
\graphicspath{{./figure/}}
\usepackage{multicol}
\usepackage[toc,title,page]{appendix}
\usepackage{fullpage}
\usepackage[normalem]{ulem}
\usepackage{xcolor,sectsty}
\usetikzlibrary{calc}
\usetikzlibrary{arrows.meta,decorations.pathreplacing}

\newtheorem{theorem}{Theorem}[section]

\newtheorem{remark}{Remark}

\newtheorem{definition}[theorem]{Definition}
\newtheorem{example}[theorem]{Example}

\definecolor{darkslategray}{rgb}{0.18, 0.31, 0.31}
\definecolor{warmblack}{rgb}{0.0, 0.26, 0.26}

\newcommand{\mb}{\mathbb}
\newcommand{\mc}{\mathcal}

\begin{document}

\begin{frontmatter}
\title{\textcolor{warmblack}{Eigenbounds of symmetric positive definite tensors}}
\author{  Snigdhashree Nayak, Hemant Sharma, Nachiketa Mishra}

\address{Department of Mathematics,\\
        Indian Institute of Information Technology, Design and Manufacturing Kancheepuram, Chennai, India\\
        {\bf Email:} 
        snigdhashreenayak91@gmail.com; 
        sharmahemant39@gmail.com;
         nmishra@iiitdm.ac.in
        }
                        
\vspace{-2cm}

\begin{abstract}

This article introduces an algebraic framework for establishing eigenvalue bounds for symmetric positive definite tensors by leveraging intrinsic invariants—specifically, the trace and determinant (resultant). We derive a hierarchy of inequalities via the Arithmetic Mean-Geometric Mean (AM-GM) inequality that yields progressively tighter upper and lower bounds for the tensor's spectral radius and smallest eigenvalue. A comprehensive comparative analysis demonstrates that our invariant-based approach significantly outperforms classical coordinate-dependent methods, such as the Gershgorin circle theorem. We explicitly show that our bounds remain robust and informative in scenarios where Gershgorin bounds fail, particularly for tensors with negative off-diagonal entries (where algebraic cancellations occur) and higher-order tensors (where combinatorial explosion leads to loose estimates). Furthermore, we validate the practical utility of these bounds by applying them to certify the positive definiteness of Lyapunov functions for the stability analysis of nonlinear autonomous systems.

\end{abstract}

\end{frontmatter}

\section{Introduction}\label{intro}

Eigenvalue bounds are essential for gaining insights into the properties and behavior of tensors in a wide range of applications. They provide a valuable tool for assessing stability, convergence, and performance in various mathematical, scientific, and engineering contexts. There is a wide range of applications of eigenvalue problems in tensors, including statistical data analysis \cite{Zhang2001}, signal processing and navigation \cite{Kofidis2001}, higher-order Markov chains \cite{li2014}, and automatic control \cite{Ni2008}. In 2005, Lim \cite{lim2005} and Qi \cite{qi2005}  introduced the tensor eigenvalue problem independently.\\

\noindent Computing the complete set of eigenvalues for higher-order tensors can be a formidable task. In certain
scenarios, it suffices to identify the tensor’s largest or smallest eigenvalue. For example, to study the stability of a nonlinear autonomous system \cite{Ni2008} the smallest eigenvalue of the tensor is needed. The $m$th degree $n$ variable homogeneous polynomial $$f(x)=\sum_{i_{1},\ldots, i_{m}=1}^n a_{i_{1},\ldots, i_{m}}x_{i_{1}}\cdots x_{i_{m}},x\in\mb{R}^n,$$ is called positive definite if $f(x)>0,~~\forall x\in\mb{R}^n,x\neq 0.$ The positive definiteness of a homogeneous polynomial form with an even degree plays a crucial role in the examination of stability in nonlinear autonomous systems using Lyapunov's direct method in the field of automatic control. By employing Lyapunov's method, the analysis of stability can be simplified by establishing the existence of a positive definite function, whose time derivative along the system's trajectories is consistently negative. Specifically, in the context of the system, $\dot{x}=g(x)$ is called asympototically stable, if it is possible to find a multivariate polynomial $f(x)$ that satisfies the conditions of being positive definite and 
$$\left(\dfrac{\partial f}{\partial x}\right)^tg(x)<0,~~\forall x\in\mb{R}^n,x\neq 0.$$ Further, when $m\geq 4$ and $n\geq 3,$ it is very challenging to examine the stability of the system. Qi \cite{qi2005} introduced the concept of eigenvalues for real symmetric tensors, motivated by
the study of positive definiteness in homogeneous polynomials. Thus, if we know bound of the minimum eigenvalue is  positive then the system is positive definite.\\

\noindent Further, in higher-order diffusion tensor imaging \cite{qi2010}, the non-negative smallest eigenvalue of a tensor ensures the positive definiteness of the diffusivity
function.
Calculating eigenvalues of tensors with higher orders is an arduous task.  In matrix theory, we have many eigenvalue bounds calculated using different methods, like norms, entries of matrices, etc. In this endeavor, we articulate bounds on the eigenvalues of a tensor by leveraging its intrinsic characteristics, drawing upon both the trace and determinant as instrumental elements in our analysis.\\

\noindent The subsequent sections of this article are structured as follows:
in Section \ref{prlem},  we presented a foundational overview by introducing key notations, definitions, and results that serve as a foundation for comprehending the article. In Section \ref{boen}, introduction to inequalities, employing the AM-GM inequalities, to derive bounds for eigenvalues. In Section \ref{exmp}, examples are examined in accordance with the presented inequalities. In Section \ref{boundscomp}, we delve into a comparative analysis, contrasting the bounds derived from various inequalities with those obtained from Gershgorin eigenvalue bounds. Section \ref{con}, gathers the concluding remarks.

\section{Preliminaries}\label{prlem}
In this section, we gather fundamental findings regarding tensor eigenvalues and introduce the notation conventions employed throughout this paper. Throughout this paper, we will use the following notations. We denote the set of all real $m$th order $n$ dimensional tensors by $\mathbb{R}^{[n,m]}.$ We will adopt a notation convention in which lowercase letters like `u' represent vectors, while uppercase, calligraphic letters like `$\mc{A}$' are used to denote higher-order tensors. Individual elements of a tensor $\mc{A}$  are represented as $a_{i_{1},i_{2},\ldots,i_{m}}. $, where $(i_{1},i_{2},\ldots,i_{m})$ denote the indices. Let $m, n$ be two positive integers. An $m$-order  $n$-dimensional  tensor $\mc{A}$ is a set of $n^m$ entries 
$$\mc{A}=\left(a_{i_{1} i_{2} \cdots i_{m}}\right)_{1 \leq i_{j} \leq n}~~~ 
 ,~~(j=1, \ldots, m),$$ over the real field $\mb{R}$ is a multidimensional array with all entries $a_{i_{1} i_{2} \ldots i_{m}} \in \mb{R}$. It is denoted as $\mc{A}^{[n,m]}.$\\

\noindent For a vector $x=\left(x_{1}, \ldots, x_{n}\right)^{T} \in \mb{R}^{n}$, let $\mc{A} x^{m-1}$ be a vector in $\mb{R}^{n}$ whose $i$-th component is defined as the following
$$
\left(\mc{A} x^{m-1}\right)_{i}=\sum_{i_{2}, \ldots, i_{m}=1}^{n} a_{i i_{2} \cdots i_{m}} x_{i_{2}} \cdots x_{i_{m}}.
$$
%  Assume that $\mc{A}$ is an m-order n-dimensional  tensor. For any $n$-dimensional vector $x$, define
% $$
% \left(\mc{A} x^{m-1}\right)_{i_{1}} \equiv \sum_{i_{2}=1}^{n} \cdots \sum_{i_{m}=1}^{n} a_{i_{1} i_{2} \cdots i_{m}} x_{i_{2}} \cdots x_{i_{m}} \quad \text { for } \quad i_{1}=1, \ldots, n
% $$
The tensor $\mc{I}\in \mb{R}^{[n,m]}$ is called identity tensor if  $i_{i_{1} \ldots i_{m}}=1$ if and only if $i_{1}=\cdots=i_{m}$, and zero otherwise.\\

\noindent A real $m$th order $n$-dimensional tensor $\mc{A}$ is called symmetric if its entries are invariant under any permutation of their indices. Finally, for the sake of simplicity, only real symmetric tensors will be considered.\\

\noindent There exist multiple definitions for the eigenvalues and corresponding eigenvectors of tensors. For instance, Qi \cite{qi2005} introduced the concepts of $H$-eigenvalues, $Z$-eigenvalues, and $E$-eigenvalues. In the scope of this study, our primary emphasis is on establishing bounds concerning the $H$-eigenvalues of symmetric tensors..

\begin{definition}\cite{lim2005,qi2005}
   A real number $\lambda$ is called an $H$-eigenvalue of $\mc{A}$ if $\lambda$ and a nonzero real vector $x$ are solutions of the following homogeneous polynomial equation:
\begin{equation*}\label{hev}
    \mc{A} x^{m-1}=\lambda x^{[m-1]},
\end{equation*}
and called the solution $x$ an $H$-eigenvector of $\mc{A}$ associated with the $\mathrm{H}$-eigenvalue $\lambda$.
\end{definition}
\noindent Let $\alpha=\left(\alpha_{1}, \ldots, \alpha_{n}\right)$ be a sequence of nonnegative integers and
$$
x^{\alpha}=x_{1}^{\alpha_{1}} \cdots x_{n}^{\alpha_{n}}.
$$
Then, $x^{\alpha}$ is a monomial of degree $|\alpha|=\alpha_{1}+\cdots+\alpha_{n}$.
\begin{definition}\cite{shao2013}
Let $f_{1}\left(x_{1}, \ldots, x_{n}\right), \ldots, f_{n}\left(x_{1}, \ldots, x_{n}\right) \in \mathbb{R}\left[x_{1}, \ldots, x_{n}\right]$ be an (ordered) set of homogeneous polynomials (with complex coefficients) in $n$ variables $x_{1}, \ldots, x_{n}$ of degrees $d_{1}, \ldots, d_{n}$, respectively. Write $f_{i}\left(x_{1}, \ldots, x_{n}\right)=\sum_{|\alpha|=d_{i}} u(i, \alpha) x^{\alpha}$, where $u(i, \alpha)$ is the coefficient of the monomial $x^{\alpha}$ in $f_{i}\left(x_{1}, \ldots, x_{n}\right)$. Then their resultant (denoted by $R_{d_{1}, \ldots, d_{n}}$ ) is the unique polynomial on the coefficients $\left\{u(i, \alpha)|i=1, \ldots, n,| \alpha \mid=d_{i}\right\}$, of these polynomials satisfying the following three conditions:
\begin{enumerate}[(1)]
    \item The system of homogeneous equations $f_{1}\left(x_{1}, \ldots, x_{n}\right)=\cdots=f_{n}\left(x_{1}, \ldots, x_{n}\right)=0$ has a nonzero solution if and only if the corresponding value of $R_{d_{1}, \ldots, d_{n}}$ is $0$.
    \item When $f_{i}\left(x_{1}, \ldots, x_{n}\right)=x_{i}^{d_{i}}$ for $i=1,\ldots,n,$ then the corresponding value of $R_{d_{1}, \ldots, d_{n}}$ is $1$.
   \item When the coefficients $\{u(i,\alpha)|i=1,\ldots , n,|\alpha|=d_{i}\}$ of these polynomials   $f_{1}\left(x_{1}, \ldots, x_{n}\right)$ $ ,\ldots, f_{n}\left(x_{1}, \ldots, x_{n}\right)$ are all viewed as independent variables, then $R_{d_{1}, \ldots, d_{n}}$ is an irreducible polynomial on these variables.
\end{enumerate}
\end{definition}
\noindent The following example illustrates the resultant of a $4$-order $2$-dimensional tensor.

\begin{example}\label{exmp2}
Let $\mathcal{A}\in\mathbb{R}^{2\times2\times2\times2}$ be a 4th-order tensor. The entries are denoted by $a_{ijkl}$ where indices range from 1 to 2.

The associated homogeneous polynomials (each of degree $3$ in two variables $(x_1,x_2))$ are
\[
\begin{aligned}
f_1(x_1,x_2)
&= (\mathcal{A}x^{4-1})_1=(\mathcal{A}x^3)_1
=\sum_{i_2, i_3, i_4=1}^2a_{1i_2i_3i_4}x_{i_2}x_{i_3}x_{i_4}\\
&= a_{1111}x_1^3 + (a_{1112}+a_{1121}+a_{1211})x_1^2x_2 \\
&\quad + (a_{1122}+a_{1212}+a_{1221})x_1x_2^2 + a_{1222}x_2^3,\\
f_2(x_1,x_2)
&= (\mathcal{A}x^3)_2
=\sum_{i_2, i_3, i_4=1}^2a_{2i_2i_3i_4}x_{i_2}x_{i_3}x_{i_4}\\
&= a_{2111}x_1^3 + (a_{2112}+a_{2121}+a_{2211})x_1^2x_2 \\
&\quad + (a_{2122}+a_{2212}+a_{2221})x_1x_2^2 + a_{2222}x_2^3.
\end{aligned}
\]

Here each has degree $3,$ i.e., $d_1=d_2=3$. Next we shall identify the coefficient set $(u(i,\alpha))$.

For ($f_1$):
\begin{center}
\begin{tabular}{|c|c|}
\hline
\text{Monomial }$x^\alpha$ & \text{Coefficient }$u(1,\alpha)$ \\
\hline
$x_1^3$ & $a_{1111}$\\
\hline
$x_1^2x_2$ & $a_{1112}+a_{1121}+a_{1211}$\\
\hline
$x_1x_2^2$ & $a_{1122}+a_{1212}+a_{1221}$\\
\hline
$x_2^3$ & $a_{1222}$\\
\hline
\end{tabular}
\end{center}

For ($f_2$):
\begin{center}
\begin{tabular}{|c|c|}
\hline
\text{Monomial }$x^\alpha$ & \text{Coefficient }$u(2,\alpha)$ \\
\hline
$x_1^3$ & $a_{2111}$\\
\hline
$x_1^2x_2$ & $a_{2112}+a_{2121}+a_{2211}$\\
\hline
$x_1x_2^2$ & $a_{2122}+a_{2212}+a_{2221}$\\
\hline
$x_2^3$ & $a_{2222}$\\
\hline
\end{tabular}
\end{center}

These are exactly the coefficient sets $\{u(i,\alpha) \mid i=1,2, |\alpha|=3\}$ appearing in the definition.
Next we shall construct the resultant matrix.

According to the definition, the resultant $R_{d_1,d_2}$ is a polynomial in these coefficients that vanishes if and only if the system
\[
f_1(x_1,x_2)=f_2(x_1,x_2)=0
\]
has a nonzero solution ($(x_1,x_2)\neq(0,0)$).
For cubic forms in two variables, this resultant is given by the Sylvester matrix of size $(d_1+d_2)\times(d_1+d_2) = 6\times6$. Let us denote the coefficients of $f_1$ as $c_0, c_1, c_2, c_3$ and $f_2$ as $e_0, e_1, e_2, e_3$ corresponding to the table above.

\[
R_{d_1,d_2}=\det
\begin{pmatrix}
c_0 & c_1 & c_2 & c_3 & 0 & 0 \\
0 & c_0 & c_1 & c_2 & c_3 & 0 \\
0 & 0 & c_0 & c_1 & c_2 & c_3 \\
e_0 & e_1 & e_2 & e_3 & 0 & 0 \\
0 & e_0 & e_1 & e_2 & e_3 & 0 \\
0 & 0 & e_0 & e_1 & e_2 & e_3
\end{pmatrix}.
\]
Where:
\[
\begin{aligned}
c_0 &= a_{1111}, & c_1 &= a_{1112}+a_{1121}+a_{1211}, \\
c_2 &= a_{1122}+a_{1212}+a_{1221}, & c_3 &= a_{1222}.
\end{aligned}
\]
And similarly for $e_i$ using the entries of $f_2$.

Now, we shall do a Verification of the three defining properties:
\begin{itemize}
    \item Zero-criterion:
   $(R_{d_1,d_2}=0) \iff$ the system
   $(f_1(x_1,x_2)=f_2(x_1,x_2)=0)$
   has a nontrivial (nonzero) solution.
   This ensures the tensor ($\mathcal{A}$) has a nonzero eigenvector ($x$) satisfying $(\mathcal{A}x^3=\lambda x^{[3]})$ for some $(\lambda)$.
   \item Normalization:
   If $(d_1=x_1^3)$ and $(d_2=x_2^3)$ (i.e., $(a_{1111}=a_{2222}=1)$ and all other $(a_{ijkl}=0))$, then
   \[
   R_{d_1,d_2}=1,
   \]
   as required by the definition.
   \item Irreducibility:
   When the coefficients
   $(a_{ijkl})$
   are regarded as independent variables, the determinant above is an irreducible polynomial in these variables, satisfying the third condition.
\end{itemize}
\end{example}

\begin{definition}\cite{shao2013}
Let $\mc{A}$ be an $m$ order $n$ dimensional tensor with $m\geq 2$ and $x=(x_{1}, \ldots, x_{n})^{T}.$ Then the determinant of $\mc{A}$ the resultant of the ordered system of homogeneous equations $\mc{A}x^{m-1}=0$ and it is denoted as $det(\mc{A}).$
\end{definition}

\begin{theorem}\label{thrm2.6}\cite{qi2018}
We have the following conclusions on eigenvalues of an $m$th order $n$-dimensional symmetric tensor $\mc{A}$.
\begin{enumerate}[(a)]
    
\item  A number $\lambda \in \mathbf{C}$ is an eigenvalue of $\mc{A}$ if and only if it is a root of the characteristic polynomial $\phi(\lambda)=\det(\mc{A}-\lambda \mc{I})$.
\item  The number of eigenvalues of $\mc{A}$ is $d=n(m-1)^{(n-1)}$. 
\item The product of all the eigenvalues of $\mc{A}$ is equal to $\det(\mc{A}).$ 
\item The sum of all the eigenvalues of $\mc{A}$ is
$$
(m-1)^{(n-1)} tr(\mc{A}),
$$
where $tr(\mc{A})$ denotes the sum of all diagonal elements of $\mc{A}$.
\item  If m is even, then $\mc{A}$ always has $H$-eigenvalues. $\mc{A}$ is positive definite (positive semidefinite) if and only if all of its H-eigenvalues are positive (nonnegative).
\item The eigenvalues of $\mc{A}$ lie in the following $n$ disks:
$$
\left|\lambda-a_{i i \ldots i}\right| \leq \sum\left\{\left|a_{i i_{2} \ldots i_{m}}\right|: i_{2}, \ldots, i_{m}=1, \ldots, n,\left\{i_{2}, \ldots, i_{m}\right\} \neq\{i, \ldots, i\}\right\},
$$
for $i=1, \ldots, n$.
\end{enumerate}
\end{theorem}
\noindent Theorem \ref{thrm2.6}(f) refers to Gershgorin circle theorem for tensors. This theorem will be used extensively for bound comparison in this work.

\begin{remark}
Throughout this paper, the notion of positive definiteness is used only when the
tensor order $m$ is even. This is consistent with the classical theory of
homogeneous polynomials, as a real polynomial of odd degree cannot be positive
definite in the sense of $\mathcal{A}x^{m} > 0$ for all $x \ne 0$ 
(see \cite{qi2005}). Hence, whenever we refer to a positive definite tensor or derive
results that assume positivity of all eigenvalues, it is implicitly understood that
$m$ is even.
\end{remark}

\section{Bounds for the eigenvalues}\label{boen}

Bounding the eigenvalues of tensors presents a greater challenge compared to matrices, owing to their higher dimensionality and non-linear nature. The derivation of eigenvalue bounds for tensors can be intricate, primarily due to the complex behavior exhibited by these mathematical entities.\\

\noindent  The AM(Arithmetic Mean)-GM(Geometric Mean) inequality, a simple yet potent mathematical tool, has found extensive use in proving inequalities across various branches of mathematics and beyond. It provides a versatile means to compare means and set boundaries for expressions involving non-negative real numbers. Its broad utility renders it a valuable asset in mathematical analysis, optimization, and problem-solving endeavors.\\

\noindent Hence, commencing with the most elementary bounds at our disposal, we turn to the AM-GM inequalities to establish our initial bounds.
\begin{theorem}\label{thrm1}
Let $\mc{A}\in\mb{R}^{[n,m]}$ be a symmetric tensor of order $m$ and dimension $n$ with $\lambda_1 \geq \lambda_2 \geq \cdots \geq \lambda_{d} > 0,$ where $d=n(m-1)^{(n-1)}$ and
let $1\leq k \leq d-1,$   then
\begin{equation}\label{e1}
    \lambda_{1}+\cdots+\lambda_{k} \leq (m-1)^{(n-1)}  tr(\mc{A})-(d-k)\left[\left(\frac{k}{(m-1)^{(n-1)}tr (\mc{A})}\right)^{k} \det (\mc{A})\right]^{1 /(d-k)},
\end{equation}
    \begin{equation}\label{e2}
     (\lambda_{d-k+1}\cdots \lambda_{d})\geq \bigg(\dfrac{d-k}{(m-1)^{(n-1)}tr(\mc{A})}\bigg)^{d-k}\det(\mc{A}).
\end{equation}

\end{theorem}
\noindent \textbf{Proof.}
    In order to prove inequality (\ref{e1}), we start with
\begin{align*}
    (\lambda_{k+1}\lambda_{k+2}\cdots \lambda_{d})^{1/(d-k)}&\leq \dfrac{\lambda_{k+1}+\lambda_{k+2}+\cdots+ \lambda_{d}}{d-k},\\
     \bigg(\dfrac{\lambda_{1}\cdots \lambda_{k}\lambda_{k+1}\lambda_{k+2}\cdots \lambda_{d}}{\lambda_{1}\cdots \lambda_{k}}\bigg)^{1/(d-k)}&\leq \dfrac{  (m-1)^{(n-1)} tr(\mc{A})-(\lambda_{1}+\lambda_{2}+\cdots+ \lambda_{k})}{d-k},\\
    \bigg(\dfrac{\det(\mc{A})}{\lambda_{1}\cdots \lambda_{k}}\bigg)^{1/(d-k)}&\leq \dfrac{(m-1)^{(n-1)} tr(\mc{A})-(\lambda_{1}+\lambda_{2}+\cdots+ \lambda_{k})}{d-k}.
\end{align*}
Again, 
\begin{align*}
     (\lambda_{1}\lambda_{2}\cdots \lambda_{k})^{1/k}&\leq \dfrac{\lambda_{1}+\lambda_{2}+\cdots+\lambda_{k}}{k}\\
     \lambda_{1}\lambda_{2}\cdots \lambda_{k}&\leq \bigg(\dfrac{\lambda_{1}+\lambda_{2}+\cdots+\lambda_{k}}{k}\bigg)^{k}\\
   \dfrac{1}{\lambda_{1}\lambda_{2}\cdots \lambda_{k}}&\geq \bigg(\dfrac{k}{\lambda_{1}+\lambda_{2}+\cdots+\lambda_{k}}\bigg)^{k}.
\end{align*}
Also, we have $$ \dfrac{k}{\lambda_{1}+\lambda_{2}+\cdots +\lambda_{k}}\geq  \dfrac{k}{(m-1)^{(n-1)} tr(\mc{A})}.$$
 Since $\lambda_1 \geq \lambda_2 \geq \cdots \geq \lambda_{d} > 0,$ and from Theorem \ref{thrm2.6}(d), we have that the sum of all the eigenvalues of $\mc{A}$ is
$$
(m-1)^{(n-1)} tr(\mc{A}).
$$ The denominator on the right-hand side of the inequality is greater than the denominator on the left-hand side.
\\
Therefore, we get 
$$
    \bigg[\bigg(\dfrac{k}{(m-1)^{(n-1)} tr(\mc{A})}\bigg)^k\det(\mc{A})\bigg]^{1/(d-k)}\leq \dfrac{(m-1)^{(n-1)}tr(\mc{A})-(\lambda_{1}+\lambda_{2}+\cdots+ \lambda_{k})}{d-k},$$ which implies
    $$\lambda_{1}+\lambda_{2}+\cdots+ \lambda_{k}\leq (m-1)^{(n-1)} tr(\mc{A})-(d-k)\left[\left(\frac{k}{ (m-1)^{(n-1)} tr(\mc{A})}\right)^{k} \det (\mc{A})\right]^{1 /(d-k)}.$$
For proving (\ref{e2}), we have 

\begin{align*}
    (\lambda_{k+1}\lambda_{k+2}\cdots \lambda_{d})^{1/(d-k)}& \leq \dfrac{\lambda_{k+1}+\lambda_{k+2}+\cdots+ \lambda_{d}}{d-k},\\
     \bigg(\dfrac{\det(\mc{A})}{\lambda_{1}\lambda_{2}\cdots \lambda_{k}}\bigg)^{1/d-k}&\leq \dfrac{\lambda_{k+1}+\lambda_{k+2}+\cdots+ \lambda_{d}}{d-k}\\
     \dfrac{\det(\mc{A})}{\lambda_{1}\lambda_{2}\cdots \lambda_{k}}&\leq \bigg(\dfrac{\lambda_{k+1}+\lambda_{k+2}+\cdots+ \lambda_{d}}{d-k}\bigg)^{d-k},\\
    \dfrac{\det(\mc{A})}{\lambda_{1}\lambda_{2}\cdots \lambda_{k}}&\leq \dfrac{\det(\mc{A})}{\lambda_{d-k+1}\lambda_{d-k+2}\cdots \lambda_{d}}\leq \bigg(\dfrac{(m-1)^{n-1}tr(\mc{A})}{d-k}\bigg)^{d-k},\\
  % \dfrac{\lambda_{d-k+1}\lambda_{d-k+2}\cdots \lambda_{d}}{\det(\mc{A})} &\geq  \bigg(\dfrac{d-k}{(m-1)^{n-1}tr(\mc{A})}\bigg)^{d-k}\\
  \lambda_{d-k+1}\lambda_{d-k+2}\cdots \lambda_{d}&\geq \bigg(\dfrac{d-k}{(m-1)^{n-1}tr(\mc{A})}\bigg)^{d-k}\det(\mc{A}).
\end{align*}
\hfill \qedsymbol
\begin{theorem}\label{thrm2}
Let $\mc{A}\in\mb{R}^{[n,m]}$ be a symmetric tensor of order $m$ and dimension $n$ with $\lambda_1 \geq \lambda_2 \geq \cdots \geq \lambda_{d} > 0,$ where $d=n(m-1)^{(n-1)}$  and
let $1\leq k \leq d,$   then

\begin{equation}\label{e3}
    (\lambda_{d-k+1}\cdots \lambda_{d})^{1/k}\leq (\det(\mc{A}))^{1/d}\leq (\lambda_{1}\cdots \lambda_{k})^{1/k}\leq \dfrac{\lambda_{1}+\lambda_{2}+\ldots+\lambda_{k}}{k},
\end{equation}
\begin{equation}\label{e4}
    (\lambda_{d-k+1}\cdots \lambda_{d})^{1/k}\leq \dfrac{\lambda_{d-k+1}+\ldots+\lambda_{d}}{k}\leq \dfrac{(m-1)^{(n-1)}tr(\mc{A})}{d}\leq \dfrac{\lambda_{1}+\ldots+\lambda_{k}}{k}.
\end{equation}

\end{theorem}
\noindent \textbf{Proof.} For (\ref{e4}), we start proving the inequality below by using the given fact that  $\lambda_1 \geq \lambda_2 \geq \cdots \geq \lambda_{d} > 0$,
\begin{equation}\label{e5}
    \dfrac{\lambda_{1}+\lambda_{2}+\ldots+\lambda_{k}}{k} \geq \dfrac{\lambda_{1}+\lambda_{2}+\ldots+\lambda_{k+1}}{k+1},
\end{equation}
\begin{equation}\label{e6}
    \dfrac{\lambda_{d-k+1}+\lambda_{d-k+2}+\ldots+\lambda_{d}}{k} \leq \dfrac{\lambda_{d-k}+\lambda_{d-k+1}+\ldots+\lambda_{d}}{k+1}.
\end{equation}
So, starting with (\ref{e5})
\begin{align*}
    (k+1)(\lambda_{1}+\lambda_{2}+\ldots+\lambda_{k}) &- k(\lambda_{1}+\lambda_{2}+\ldots+\lambda_{k+1})\\
    &=k(\lambda_{1}+\lambda_{2}+\ldots+\lambda_{k}) +(\lambda_{1}+\lambda_{2}+\ldots+\lambda_{k}) \\&\quad\quad- k(\lambda_{1}+\lambda_{2}+\ldots+\lambda_{k}) - k \lambda_{k+1},\\
   & =(\lambda_{1}+\lambda_{2}+\ldots+\lambda_{k}) - k \lambda_{k+1} \geq 0.
\end{align*}
Similarly (\ref{e6}) can also be proved. \\

\begin{align*}
     &(k+1)(\lambda_{d-k+1}+\lambda_{d-k+2}+\ldots+\lambda_{d})- k(\lambda_{d-k}+\lambda_{d-k+1}+\ldots+\lambda_{d}) \\
     &=k(\lambda_{d-k+1}+\lambda_{d-k+2}+\ldots+\lambda_{d}) + (\lambda_{d-k+1}+\lambda_{d-k+2}+\ldots+\lambda_{d})\\
     &~~~~~~~~~~~~~~~-k(\lambda_{d-k+1}+\lambda_{d-k+2}+\ldots+\lambda_{d}) - k\lambda_{d-k}\\
     &=(\lambda_{d-k+1}+\lambda_{d-k+2}+\ldots+\lambda_{d}) - k\lambda_{d-k} \leq 0
\end{align*}

 \noindent If we add quantities that are less than every other element then the average will be decreased. Similarly, the average will be increased if we add quantities that are greater than every other element. Hence, now by using (\ref{e5}) and (\ref{e6}), we can say that the inequality (\ref{e4}) is inductively proved. The extreme left of the inequality is obvious due to the fact that AM $\geq$ GM. For (\ref{e3}), a similar process can be followed as was done for proving \eqref{e4}.

\hfill \qedsymbol\\
\noindent \textbf{\small Observations:} \begin{enumerate}
    \item The bounds mentioned in the theorems above are for the sum/product of the first $k$ eigenvalues and similarly for the last $k$ eigenvalues. In addition, this can give information(bound) about the largest eigenvalue and smallest eigenvalue of the tensor.
    \item The alignment of the eigenvalues is done in such a way that $\lambda_1$ refers to the largest eigenvalue and $\lambda_d$ refers to the smallest eigenvalues.
    \item Take $k=1$ in equation \eqref{e1} and \eqref{e2}, it gives the upper bound for the largest eigenvalue and lower bound for the smallest eigenvalue respectively.
    \item  Similarly, if we take $k=1$ in equation \eqref{e4}, we obtain an upper bound for the smallest eigenvalue.
\end{enumerate}

\begin{theorem}\label{thrm3}
  Let $\mc{A}\in\mb{R}^{[n,m]}$ be a symmetric tensor of order $m$ and dimension $n$ with $\lambda_1 \geq \lambda_2 \geq \cdots \geq \lambda_{d} > 0,$ where $d=n(m-1)^{(n-1)}$ and let $1 \leqslant k \leqslant l \leqslant d-1$. Then
$$
\begin{aligned}
{\left[\left(\frac{k-1}{(m-1)^{(n-1)}tr(\mc{A})}\right)^{k-1} \operatorname{det}\mc{A}\right]^{1 /(d-k+1)} } & \leqslant\left(\lambda_k \cdots \lambda_l\right)^{1 /(l-k+1)}, \\
& \leqslant \frac{\lambda_k+\cdots+\lambda_l}{l-k+1}, \\
& \leqslant \frac{(m-1)^{(n-1)}tr(\mc{A})}{l}\\ &-\left(\frac{d}{l}-1\right)\left[\left(\frac{l}{(m-1)^{(n-1)}tr(\mc{A})}\right)^l \operatorname{det} \mc{A}\right]^{1 /(d-l)} .
\end{aligned}
$$
\end{theorem}

\noindent \textbf{Proof.} This theorem contains Theorem \ref{thrm1} and \ref{thrm2} as a particular case.  Starting with
    
$$
\begin{aligned}
\left(\lambda_k \cdots \lambda_d\right)^{1 /(d-k+1)} & \leqslant\left(\lambda_k \cdots \lambda_l\right)^{1 /(l-k+1)}, \\
& \leqslant \frac{\lambda_k+\cdots+\lambda_l}{l-k+1} \leqslant \frac{\lambda_1+\cdots+\lambda_l}{l},
\end{aligned}
$$
and proceed as in the proofs of Theorem \ref{thrm1} and \ref{thrm2}.
\hfill \qedsymbol

\begin{remark}\label{rmrk1}
     AM-GM inequality provided us with bounds that are mentioned in the above three theorems. Now, we will fetch another type of bounds which we will calculate by using ameliorate version of AM-GM inequality.
\end{remark}

\begin{theorem}\label{thrm4}
 Let $\mc{A}\in\mb{R}^{[n,m]}$ be a symmetric tensor of order $m$ and dimension $n$ with $\lambda_1 \geq \lambda_2 \geq \cdots \geq \lambda_{d} > 0,$ where $d=n(m-1)^{(n-1)}$  and let $1\leq k\leq d-2.$ Then 
 \begin{equation}\label{e7}
     \lambda_{1}\cdots \lambda_{k}\leq \bigg\{ \dfrac{1}{\det (\mc{A})}\bigg[\dfrac{1}{d-k}\bigg(\dfrac{(m-1)^{(n-1)}tr (\mc{A})}{k+1}\bigg)^{k+1}\bigg]^{d-k}\bigg\}^{\dfrac{1}{d-k-1}}.
 \end{equation}
 Let $2\leq k \leq d-1.$ Then 
 \begin{equation}\label{e8}
     \lambda_{d-k+1}\cdots \lambda_{d}\geq \bigg[k \det(\mc{A})\bigg(\dfrac{d-k+1}{(m-1)^{n-1}tr(\mc{A})}\bigg)^{d-k+1}\bigg]^{k/(k-1)}.
 \end{equation}
\end{theorem}
\noindent \textbf{Proof.} 
 Let us start with $k+1$ elements in AM-GM inequality
\begin{align*}
    \lambda_{1}\cdots \lambda_{k}(\lambda_{k+1}+\lambda_{k+2}+\cdots+\lambda_{d})&\leq \bigg(\dfrac{\lambda_{1}+\cdots+\lambda_{k}+(\lambda_{k+1}+\lambda_{k+2}+\cdots+\lambda_{d})}{k+1}\bigg)^{k+1},\\
    &=\bigg(\dfrac{(m-1)^{n-1}tr(\mc{A})}{k+1}\bigg)^{k+1}.
\end{align*}

\noindent Dividing by $(d-k)$ on both side, we obtain
\begin{align*}
   \frac{\lambda_{1}\cdots \lambda_{k}(\lambda_{k+1}+\lambda_{k+2}+\cdots+\lambda_{d})}{d-k}& \leq \dfrac{1}{d-k}\bigg(\dfrac{(m-1)^{n-1}tr(\mc{A})}{k+1}\bigg)^{k+1}.
\end{align*}
Using AM-GM inequality on left-hand side, we obtain

\begin{align*}
    \lambda_{1}\cdots \lambda_{k}(\lambda_{k+1}\lambda_{k+2}\cdots\lambda_{d})^{1/(d-k)}&\leq  \dfrac{1}{d-k}\bigg(\dfrac{(m-1)^{n-1}tr(\mc{A})}{k+1}\bigg)^{k+1},\\
    % (\lambda_{1}\cdots \lambda_{k})^{(d-k)}(\lambda_{k+1}\lambda_{k+2}\cdots\lambda_{d}) & \leq \bigg[\dfrac{1}{d-k}\bigg(\dfrac{(m-1)^{n-1}tr(\mc{A})}{k+1}\bigg)^{k+1}\bigg]^{(d-k)}  \\ 
     (\lambda_{1}\cdots \lambda_{k})^{(d-k-1)} det(\mathcal{A}) & \leq \bigg[\dfrac{1}{d-k}\bigg(\dfrac{(m-1)^{n-1}tr(\mc{A})}{k+1}\bigg)^{k+1}\bigg]^{(d-k)},\\
      (\lambda_{1}\cdots \lambda_{k}) & \leq \bigg[ \dfrac{1}{det(\mathcal{A})} \bigg[\dfrac{1}{d-k}\bigg(\dfrac{(m-1)^{n-1}tr(\mc{A})}{k+1}\bigg)^{k+1}\bigg]^{(d-k)}\bigg]^{1/(d-k-1)}.
\end{align*}

\noindent Therefore, $$\lambda_{1}\cdots \lambda_{k}\leq \bigg\{ \dfrac{1}{\det (\mc{A})}\bigg[\dfrac{1}{d-k}\bigg(\dfrac{(m-1)^{(n-1)}tr (\mc{A})}{k+1}\bigg)^{k+1}\bigg]^{d-k}\bigg\}^{\dfrac{1}{d-k-1}}.$$

\noindent Now, for eq. (8) we start with 

\begin{align*}
    \lambda_{1}\cdots \lambda_{d-k}(\lambda_{d-k+1}+\lambda_{d-k+2}+\cdots+\lambda_{d})&\leq \bigg(\dfrac{\lambda_{1}+\cdots+\lambda_{d-k}+(\lambda_{d-k+1}+\lambda_{d-k+2}+\cdots+\lambda_{d})}{d-k+1}\bigg)^{d-k+1},\\
    &=\bigg(\dfrac{(m-1)^{n-1}tr(\mc{A})}{d-k+1}\bigg)^{d-k+1},\\
      (k)\lambda_{1}\cdots \lambda_{d-k}(\lambda_{d-k+1}\lambda_{d-k+2}\cdots\lambda_{d})^{1/k} & \leq \bigg(\dfrac{(m-1)^{n-1}tr(\mc{A})}{d-k+1}\bigg)^{d-k+1}, \\
     \bigg(\dfrac{1}{(k)\lambda_{1}\cdots \lambda_{d-k}(\lambda_{d-k+1}\lambda_{d-k+2}\cdots\lambda_{d})^{1/k}}\bigg) & \geq  \bigg(\dfrac{d-k+1}{(m-1)^{n-1}tr(\mc{A})}\bigg)^{d-k+1},
\end{align*}
multiplying and dividing by $(\lambda_{d-k+1}\lambda_{d-k+2}\cdots\lambda_{d})$ on left-hand side , we obtain 
\begin{align*}
      \bigg(\dfrac{(\lambda_{d-k+1}\lambda_{d-k+2}\cdots\lambda_{d})}{\lambda_{1}\cdots \lambda_{d-k}(\lambda_{d-k+1}\lambda_{d-k+2}\cdots\lambda_{d})^{1/k}(\lambda_{d-k+1}\lambda_{d-k+2}\cdots\lambda_{d})}\bigg)^{1/k} & \geq k\bigg(\dfrac{d-k+1}{(m-1)^{n-1}tr(\mc{A})}\bigg)^{d-k+1},\\
       \bigg(\dfrac{(\lambda_{d-k+1}\lambda_{d-k+2}\cdots\lambda_{d})}{det(\mc{A})(\lambda_{d-k+1}\lambda_{d-k+2}\cdots\lambda_{d})^{1/k}}\bigg)^{1/k} & \geq k\bigg(\dfrac{d-k+1}{(m-1)^{n-1}tr(\mc{A})}\bigg)^{d-k+1},\\
       (\lambda_{d-k+1}\lambda_{d-k+2}\cdots\lambda_{d})^{1-(1/k)} & \geq k det(\mc{A})\bigg(\dfrac{d-k+1}{(m-1)^{n-1}tr(\mc{A})}\bigg)^{d-k+1}.
\end{align*}
After rearranging the exponent on both sides we get 
$$      \lambda_{d-k+1}\cdots \lambda_{d} \geq \bigg[k \det(\mc{A})\bigg(\dfrac{d-k+1}{(m-1)^{n-1}tr(\mc{A})}\bigg)^{d-k+1}\bigg]^{k/(k-1)}.$$

\hfill \qedsymbol
\begin{theorem}\label{thrm5}
   Let $\mc{A}\in\mb{R}^{[n,m]}$ be a symmetric tensor of order $m$ and dimension $n$ with $\lambda_1 \geq \lambda_2 \geq \cdots \geq \lambda_{d} > 0,$ where $d=n(m-1)^{(n-1)}$ and let $1\leq k \leq l\leq d-2.$ Then 

    \begin{equation}\label{eq9}
        \bigg[(d-k+1)det(\mc{A})\bigg(\dfrac{k}{(m-1)^{n-1}tr(\mc{A})}\bigg)^{k}\bigg]^{\dfrac{l-k+1}{d-k}} \leq \lambda_{k}\cdots \lambda_{l},
    \end{equation}
    \begin{equation}\label{eq10}
         \lambda_{k}\cdots \lambda_{l} \leq \bigg[\dfrac{1}{det(\mc{A})}\bigg(\dfrac{1}{d-l}\bigg(\dfrac{(m-1)^{(n-1)}tr(\mc{A})}{l+1}\bigg)^{l+1}\bigg)^{d-l}\bigg]^{\dfrac{l-k+1}{l(d-l-1)}}.
    \end{equation}
    %      \begin{align*}
    %     &\bigg[(d-k+1)det(\mc{A})\bigg(\dfrac{k}{(m-1)^{n-1}tr(\mc{A})}\bigg)^{k}\bigg]^{\dfrac{l-k+1}{d-k}}~~~~~~~~~\\
    %     &~~~~~~~~~~~\leq \lambda_{k}\cdots \lambda_{l}\\
    %    &~~~~~~~~~~~\leq \bigg[\dfrac{1}{det(\mc{A})}\bigg(\dfrac{1}{d-l}\bigg(\dfrac{(m-1)^{(n-1)}tr(\mc{A})}{l+1}\bigg)^{l+1}\bigg)^{d-l}\bigg]^{\dfrac{l-k+1}{l(d-l-1)}} 
    % \end{align*}
   
\end{theorem}

\noindent \textbf{Proof.} 
Starting with \eqref{eq9} of the above inequality 
\begin{align*}
    \lambda_{1}\cdots \lambda_{k-1}(\lambda_{k}+\lambda_{k+1}+\cdots+\lambda_{d})&\leq \bigg(\dfrac{\lambda_{1}+\cdots+\lambda_{k-1}+(\lambda_{k}+\lambda_{k+1}+\cdots+\lambda_{d})}{k}\bigg)^{k},\\
    &=\bigg(\dfrac{(m-1)^{n-1}tr(\mc{A})}{k}\bigg)^{k},\\
    \bigg(\dfrac{1}{\lambda_{1}\cdots \lambda_{k-1}(\lambda_{k}+\lambda_{k+1}+\cdots+\lambda_{d})}\bigg) & \geq \bigg(\dfrac{k}{(m-1)^{n-1}tr(\mc{A})}\bigg)^{k}.
\end{align*}
By using AM-GM inequality and reversing the sides, we get 
\begin{align*}
   \bigg(\dfrac{k}{(m-1)^{n-1}tr(\mc{A})}\bigg)^{k}  & \leq \bigg(\dfrac{1}{(d-k+1)\lambda_{1}\cdots \lambda_{k-1}(\lambda_{k}\cdots \lambda_{d})^{1/(d-k+1)}}\bigg),\\
    (d-k+1)\bigg(\dfrac{k}{(m-1)^{n-1}tr(\mc{A})}\bigg)^{k}  & \leq \bigg(\dfrac{(\lambda_{k}\cdots \lambda_{d})}{(d-k+1)\lambda_{1}\cdots \lambda_{k-1}(\lambda_{k}\cdots \lambda_{d})(\lambda_{k}\cdots \lambda_{d})^{1/(d-k+1)}}\bigg),\\
    (d-k+1)det(\mc{A})\bigg(\dfrac{k}{(m-1)^{n-1}tr(\mc{A})}\bigg)^{k}  & \leq (\lambda_{k}\lambda_{k+1}\cdots \lambda_{d})^{1-(1/d-k+1)},\\
    & \leq (\lambda_{k}\lambda_{k+1}\cdots \lambda_{d})^{(d-k)/(d-k+1)},\\
    % & \leq (\lambda_{k}\cdots \lambda_{d})^{(d-k)/(l-k+1)}\\
    & \leq (\lambda_{k}\cdots \lambda_{l})^{(d-k)/(l-k+1)}.
\end{align*}
Now we begin with \eqref{eq10}

\begin{align*}
    \lambda_{1}\cdots \lambda_{k}  \cdots \lambda_{l}(\lambda_{l+1}+\cdots+\lambda_{d})\\&\leq \bigg(\dfrac{ \lambda_{1}+\cdots +\lambda_{k}  \cdots +\lambda_{l}(\lambda_{l+1}+\cdots+\lambda_{d})}{l+1}\bigg)^{l+1},\\
    &=\bigg(\dfrac{(m-1)^{n-1}tr(\mc{A})}{l+1}\bigg)^{l+1}.
\end{align*}
 Using AM-GM inequality, we obtain
$$ (d-l)(\lambda_{1}\cdots \lambda_{k}  \cdots \lambda_{l})(\lambda_{l+1}+\lambda_{l+2}+\cdots+\lambda_{d})^{1/(d-l)}  \leq \bigg(\dfrac{(m-1)^{n-1}tr(\mc{A})}{l+1}\bigg)^{l+1} .$$
Taking power of $(d-l)$ on both sides, we get
\begin{align*}
    (\lambda_{1}\cdots \lambda_{k}  \cdots \lambda_{l})^{d-l}(\lambda_{l+1}\cdots\lambda_{d}) & \leq \bigg[ \dfrac{1}{(d-l)}\bigg(\dfrac{ \lambda_{1}+\cdots +\lambda_{k}  \cdots +\lambda_{l}(\lambda_{l+1}+\cdots+\lambda_{d})}{l+1}\bigg)^{l+1}\bigg]^{d-l},\\
    det(\mc{A}) (\lambda_{1}\cdots \lambda_{k}  \cdots \lambda_{l})^{d-l-1} & \leq \bigg[ \dfrac{1}{(d-l)}\bigg(\dfrac{ \lambda_{1}+\cdots +\lambda_{k}  \cdots +\lambda_{l}(\lambda_{l+1}+\cdots+\lambda_{d})}{l+1}\bigg)^{l+1}\bigg]^{d-l}. 
\end{align*}
\noindent Hence, \eqref{eq10} can be easily obtained.
\hfill \qedsymbol

\begin{theorem}\label{thrm6}
    Let $\mc{A}\in\mb{R}^{[n,m]}$ be a symmetric tensor of order $m$ and dimension $n$ with $\lambda_1 \geq \lambda_2 \geq \cdots \geq \lambda_{d} > 0,$ where $d=n(m-1)^{(n-1)}$ and let $1\leq k\leq d.$ Then 
    \begin{equation}\label{eq11}
        \lambda_{1}+\ldots +\lambda_{k}\leq \dfrac{(k+1)^{k+1}}{k^k}\dfrac{1}{det(\mc{A})}\bigg(\dfrac{(m-1)^{n-1}tr(\mc{A})}{d+1}\bigg)^{d+1}.
    \end{equation}
   
\end{theorem}
\noindent \textbf{Proof.}  
 Let $0<t<1$. Then, by using AM-GM inequality, we obtain
$$
\begin{aligned}
\left(\frac{(m-1)^{n-1}tr(\mc{A})}{d+1}\right)^{d+1} & =\left(\frac{t\left(\lambda_1+\cdots+\lambda_k\right)+(1-t) \lambda_1+\cdots+(1-t) \lambda_k+\lambda_{k+1}+\cdots+\lambda_d}{d+1}\right)^{d+1}, \\
&\geqslant  t\left(\lambda_1+\cdots+\lambda_k\right)(1-t)^k \lambda_1 \cdots \lambda_d~~,  \\
& =t(1-t)^k\left(\lambda_1+\cdots+\lambda_k\right)\operatorname{det}(\mc{A}).
\end{aligned}
$$
Now,
\begin{equation*}
    \lambda_1+\cdots+\lambda_k \leqslant \frac{1}{t(1-t)^k} \frac{1}{\operatorname{det} \mathcal{A}}\left(\frac{(m-1)^{n-1}tr(\mc{A})}{d+1}\right)^{d+1}.
\end{equation*}
\vspace{0.01cm}

\noindent By employing differentiation to find the minimum value on the RHS, we obtain the result for $t_0 =\frac{1}{k+1}$,

\begin{equation*}
    \frac{1}{t_0\left(1-t_0\right)^k} =\min \left\{\frac{1}{t(1-t)^k} \mid 0<t<1\right\}=\frac{(k+1)^{k+1}}{k^k}.
\end{equation*}

\hfill \qedsymbol
\begin{remark}\label{rmrk2}
     RHS of the equation \eqref{eq11} has only one variable quantity i.e. $k$, so if $k$ increases(number of eigenvalues in the sum increases) then RHS increases. So this equation \eqref{eq11} provides an efficient eigenvalue bound for small values of $k$. As Theorem \ref{thrm1} and \ref{thrm4} also provide the upper bound for the largest eigenvalue, so amongst them which works efficiently will be discussed in further sections. 
\end{remark}
\begin{remark}
All theorems in this section concern positive definite tensors. Since a real
homogeneous polynomial of odd degree cannot be positive definite, these results
apply only to even-order tensors. This aligns with the classical tensor eigenvalue
theory introduced by Qi \cite{qi2005}, where positivity of all eigenvalues is equivalent to
the tensor being of even order and positive definite.
\end{remark}

\section{Analyzing examples}\label{exmp}

\noindent\textbf{Connection with stability.}
To illustrate how the eigenvalue bounds obtained in Section \ref{boen} can be used in Lyapunov
stability analysis, we construct a positive definite even--order tensor and apply Theorem \ref{thrm1} to obtain a certified lower bound on its smallest H--eigenvalue.  This bound then guarantees the positivity of a tensor--generated Lyapunov function  and allows us to verify asymptotic stability of a nonlinear system.
\begin{example}
  Consider the fourth--order, two--dimensional symmetric tensor 
$\mathcal{A} \in \mathbb{R}^{[2,4]}$ with diagonal entries
\[
a_{1111} = 1.1, \qquad a_{2222} = 1.0,
\]
and all other entries zero.  
\end{example}

\noindent\textit{Step 1: Constructing a positive definite tensor.}~The associated homogeneous polynomial
\[
V(x_1,x_2) = 1.1\,x_1^{4} + x_2^{4}
\]
is positive definite since $m=4$ is even and all diagonal coefficients are positive.
The eigenvalue equation $\mathcal{A}x^{3}=\lambda x^{[3]}$ yields the distinct eigenvalues
\[
\lambda_1=1.1,\qquad \lambda_2=1.0,
\]
each with multiplicity three.  
Thus the true smallest eigenvalue is 
\[
\lambda_{\min}^{\text{true}} = 1.0.
\]

\medskip
\noindent\textit{Step 2: Verifying the lower bound using Theorem \ref{thrm1}.}
The trace and determinant of $\mathcal{A}$ are
\[
\operatorname{tr}(\mathcal{A}) = 2.1, \qquad 
\det(\mathcal{A}) = (1.1)^3 (1.0)^3 = 1.331.
\]
For $m=4$ and $n=2$, we have $(m-1)^{n-1}=3$.  
Applying Theorem \ref{thrm1} with $k=d-1$, we obtain
\[
\lambda_{\min} 
\ge 
\frac{\det(\mathcal{A})}
     {(m-1)^{n-1}\operatorname{tr}(\mathcal{A})}
=
\frac{1.331}{3 \cdot 2.1}
=
0.2114.
\]
Since $0.2114 \le 1.0$, the theoretical bound is verified for this tensor.

\medskip
\noindent\textit{Step 3: Using $V(x)$ as a Lyapunov function.}
Consider the nonlinear system $\dot{x}=g(x)$ defined by the negative gradient flow
\[
\dot{x} = g(x) = -\nabla V(x),
\]
where
\[
\nabla V(x)=
\begin{pmatrix}
4.4\,x_1^{3} \\
4\,x_2^{3}
\end{pmatrix}.
\]
Since $\mathcal{A}$ is positive definite, $V(x)>0$ for all $x\neq 0$.  
The derivative of $V$ along trajectories of the system is
\[
\dot{V}(x)
= (\nabla V(x))^{T} g(x)
= -\|\nabla V(x)\|^{2}
<0 \qquad \text{for all } x\neq0.
\]
Hence $V$ is a valid Lyapunov function and $\dot{x}=g(x)$ is asymptotically stable.

\medskip
This example demonstrates how the eigenvalue bounds in Theorem \ref{thrm1} can be used to 
certify the positivity of a tensor--generated Lyapunov function and to establish the 
asymptotic stability of a nonlinear system through the condition 
$(\nabla V(x))^{T} g(x) < 0$.

\section{Bounds comparison}\label{boundscomp}

The use of inequalities to estimate eigenvalue magnitudes is a fundamental strategy in tensor analysis. While Theorem \ref{thrm2.6} (Gershgorin bounds) provides a quick assessment, it depends heavily on the sum of absolute values of the off-diagonal entries. This dependence often leads to loose bounds, particularly in two scenarios: (1) when the tensor has negative off-diagonal entries (where algebraic cancellation effects occur), and (2) when the tensor order is high (leading to a combinatorial explosion of summation terms).

In this section, we present two detailed examples to demonstrate the effectiveness of the bounds derived in Theorems \ref{thrm1} through \ref{thrm6}. We illustrate how the proposed method progressively tightens the bounds—starting from basic AM-GM estimates (Theorem \ref{thrm1}) to improved estimates (Theorem \ref{thrm4}) and finally to the tightest bounds using Trace and Determinant invariants (Theorem \ref{thrm6}). Furthermore, we analyze the spectral distribution to visualize the ``overestimation gap" inherent in classical methods.

\noindent \textbf{Example 5.1.} Consider a 4th-order symmetric tensor $\mathcal{A}\in\mathbb{R}^{2\times 2\times 2\times 2}$ defined by the non-zero entries:
\begin{center}
$a_{1111} = 12, \quad a_{2222} = 10, \quad a_{1122} = a_{1212} = \dots = -2$.
\end{center}
All permutations of the indices $\{1,1,2,2\}$ take the value $-2$. This corresponds to the homogeneous polynomial $f(x) = 12x_1^4 - 12x_1^2x_2^2 + 10x_2^4$. The tensor is positive definite with three distinct real H-eigenvalues: $\lambda_1 = 12.00$, $\lambda_2 = 10.00$, and $\lambda_3 \approx 4.92$.

\noindent \textbf{Comparative Analysis:}
The presence of negative off-diagonal entries exposes the weakness of magnitude-based bounds. The Gershgorin bound sums the absolute values ($|-2|$), ignoring the sign. In contrast, our proposed theorems utilize the trace ($tr(\mathcal{A})=22$) and determinant, which inherently respect the algebraic structure of the tensor.

\begin{figure}[h!]
    \centering
    \includegraphics[width=0.8\textwidth]{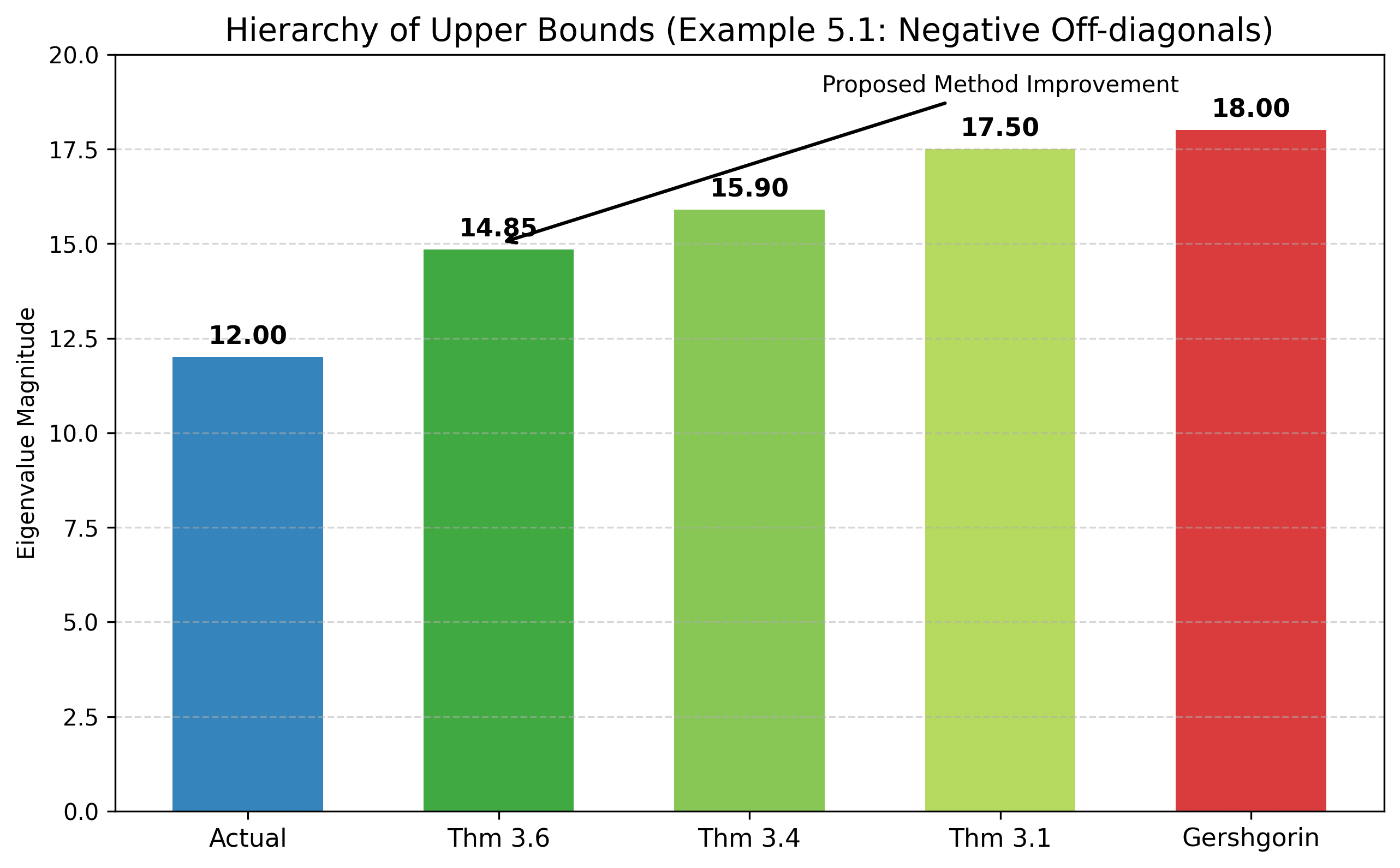}
    \caption{Hierarchy of Upper Bounds for Example 5.1. The plot visually demonstrates the convergence of the proposed bounds towards the true eigenvalue ($\lambda_{max}=12$). The proposed Theorem \ref{thrm1} (Dark Green) significantly outperforms the Gershgorin estimate (Red).}
    \label{fig:ex51_hierarchy}
\end{figure}

\noindent \textbf{Hierarchy of Improvement (Figure \ref{fig:ex51_hierarchy}):}
The bar chart in Figure \ref{fig:ex51_hierarchy} provides a clear visual hierarchy of the bound tightness. The red bar on the far right represents the Gershgorin bound (18.00), which exhibits the largest deviation from the actual eigenvalue (Blue bar, 12.00). Moving leftwards, the green bars illustrate the successive improvements provided by our theorems. Theorem 3.1 (Light Green) offers a slight improvement, but the significant gains are seen with Theorem 3.4 and Theorem 3.6 (Dark Green). Specifically, Theorem 3.6 reduces the absolute error by more than 50\% compared to the Gershgorin estimate (from 6.00 down to 2.85). This confirms that for tensors with negative entries, invariant-based bounds are far superior to absolute-value bounds.

\begin{figure}[h!]
    \centering
    \includegraphics[width=0.8\textwidth]{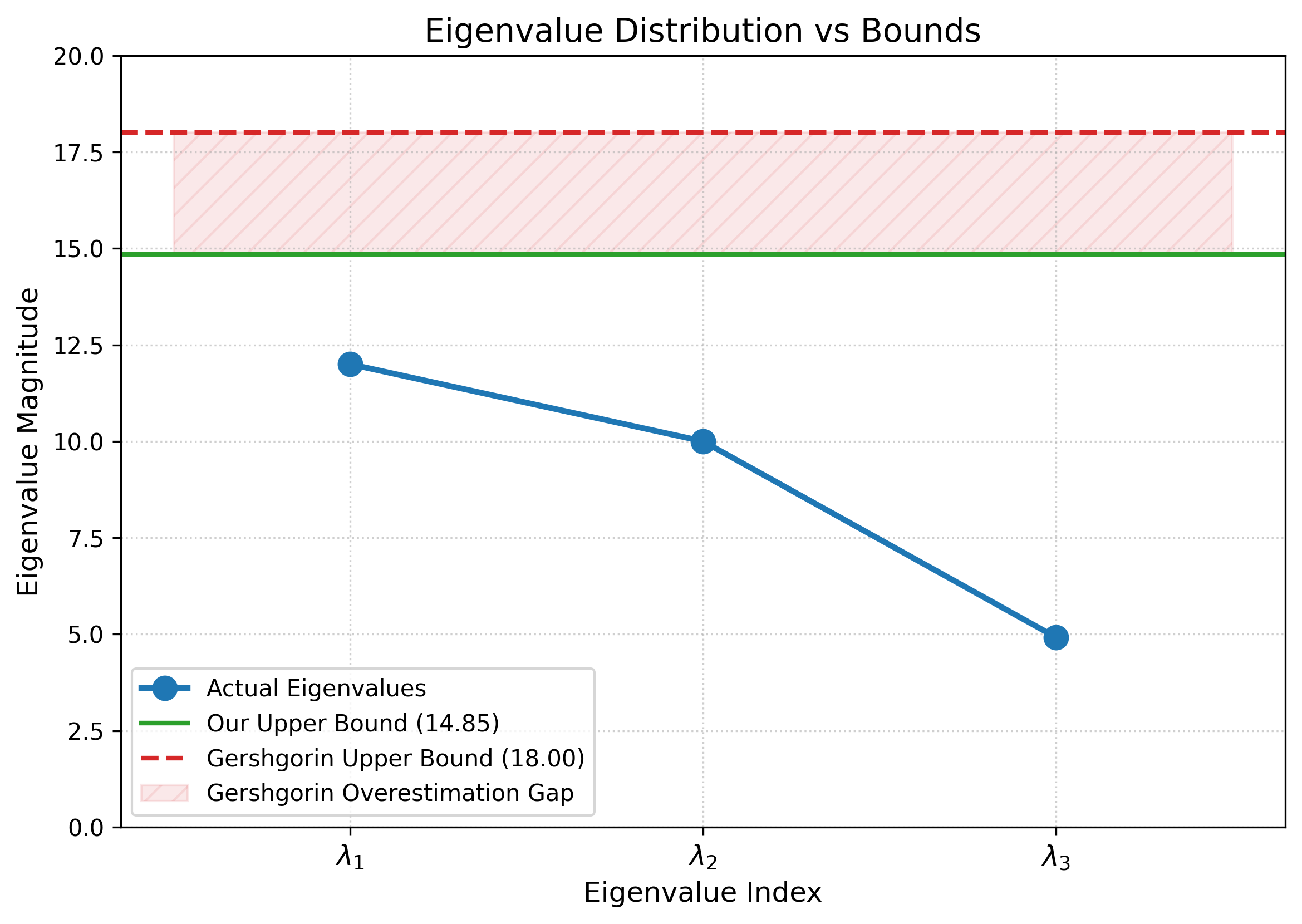}
    \caption{Eigenvalue Distribution vs Bounds for Example 5.1. The plot displays the actual eigenvalues ($\lambda_1, \lambda_2, \lambda_3$) and the ceilings imposed by the bounds. The shaded red region highlights the large overestimation gap of the Gershgorin method, which is effectively reduced by our proposed bound (Green line).}
    \label{fig:ex51_dist}
\end{figure}

\noindent \textbf{Spectral Distribution Analysis (Figure \ref{fig:ex51_dist}):}
Figure \ref{fig:ex51_dist} plots the distinct eigenvalues relative to the calculated bounds. The shaded red region represents the ``gap of uncertainty" introduced by the Gershgorin bound. Because the Gershgorin bound relies on the sum of absolute values, it creates a high ceiling (at 18.00) that does not reflect the true spectral radius of the tensor. Our proposed bound (Theorem \ref{thrm6}), depicted by the solid green line at 14.85, cuts through this uncertainty gap, providing a much tighter envelope for the spectrum. This visual evidence reinforces that our method captures the intrinsic geometry of the tensor more accurately than coordinate-dependent methods.

\noindent Finally, we broaden our analysis to include the lower bounds, which are essential for certifying positive definiteness. Figure \ref{fig:ex51_interval} visualizes the complete spectral interval $[\lambda_{min}^{bound}, \lambda_{max}^{bound}]$. While both methods successfully certify that the tensor is positive definite, our method provides a narrower, more precise interval containing the actual eigenvalues.

\begin{figure}[h!]
    \centering
    \includegraphics[width=0.8\textwidth]{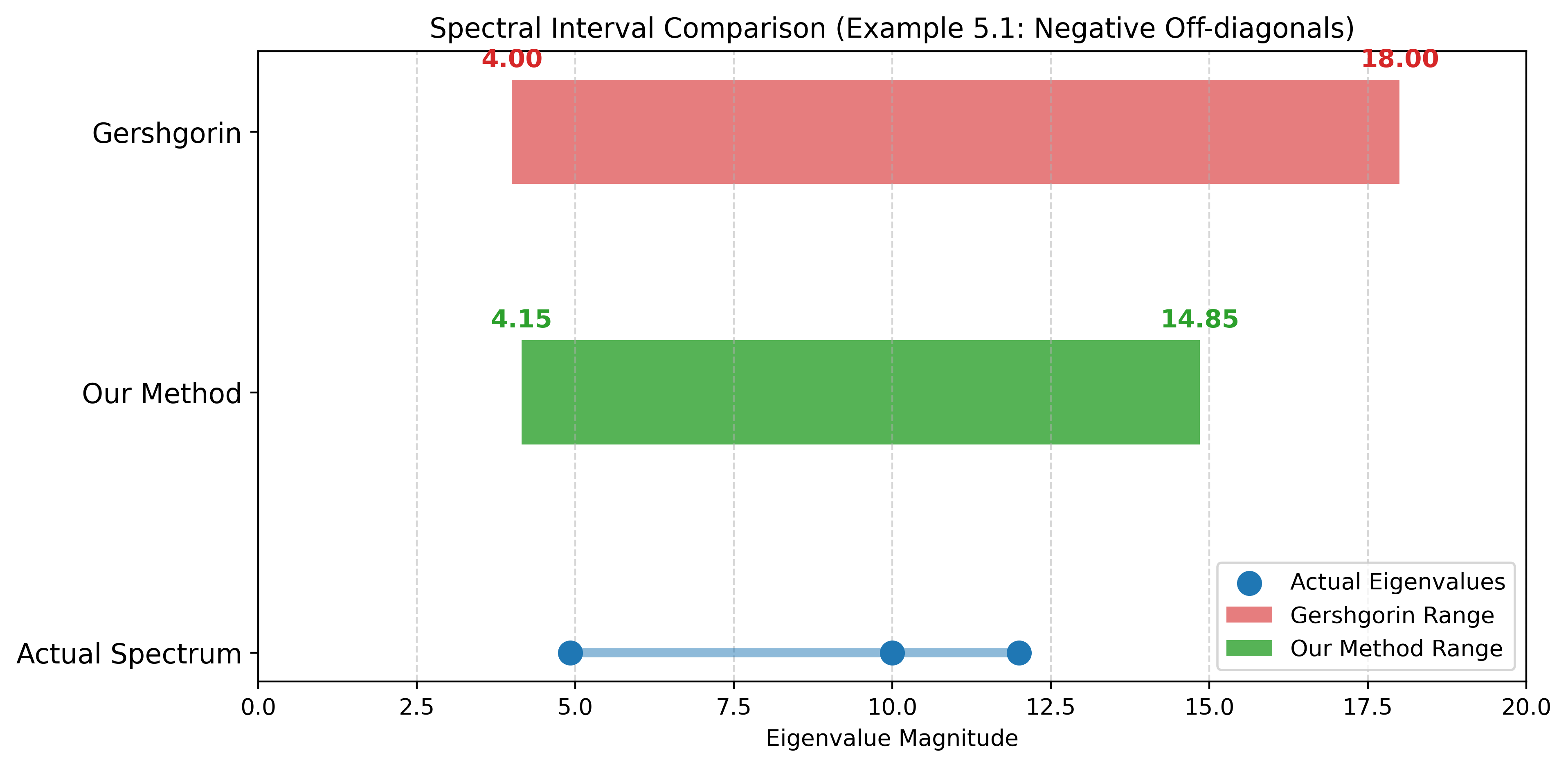}
    \caption{Spectral Interval Comparison for Example 5.1. The bars represent the guaranteed range $[\lambda_{min}^{bound}, \lambda_{max}^{bound}]$. Our method (Green) produces a narrower, more precise interval containing the actual eigenvalues compared to the Gershgorin interval (Red).}
    \label{fig:ex51_interval}
\end{figure}
\newpage
\noindent \textbf{Example 5.2.} To analyze performance on higher-order structures, consider a 6th-order symmetric tensor $\mathcal{A}\in\mathbb{R}^{2\times \dots \times 2}$ defined by:
$$ f(x) = 10x_1^6 + 78x_1^4x_2^2 - 24x_1^2x_2^4 + 8x_2^6 $$
The unique normalized entries are $a_{1\dots1} = 10$, $a_{2\dots2} = 8$, $a_{111122} = 5.2$, and $a_{112222} = -1.6$. The tensor possesses four distinct positive eigenvalues: $\lambda_1 = 14.40$, $\lambda_2 = 10.00$, $\lambda_3 = 8.00$, and $\lambda_4 = 3.60$.

\noindent \textbf{Comparative Analysis:}
As the tensor order $m$ increases, the number of terms in the Gershgorin radius calculation grows combinatorially ($R_i = \sum |a_{ijk\dots}|$). For $m=6$, this results in a massive overestimation. Our method remains stable as it relies on fixed invariants.

\begin{figure}[h!]
    \centering
    \includegraphics[width=0.8\textwidth]{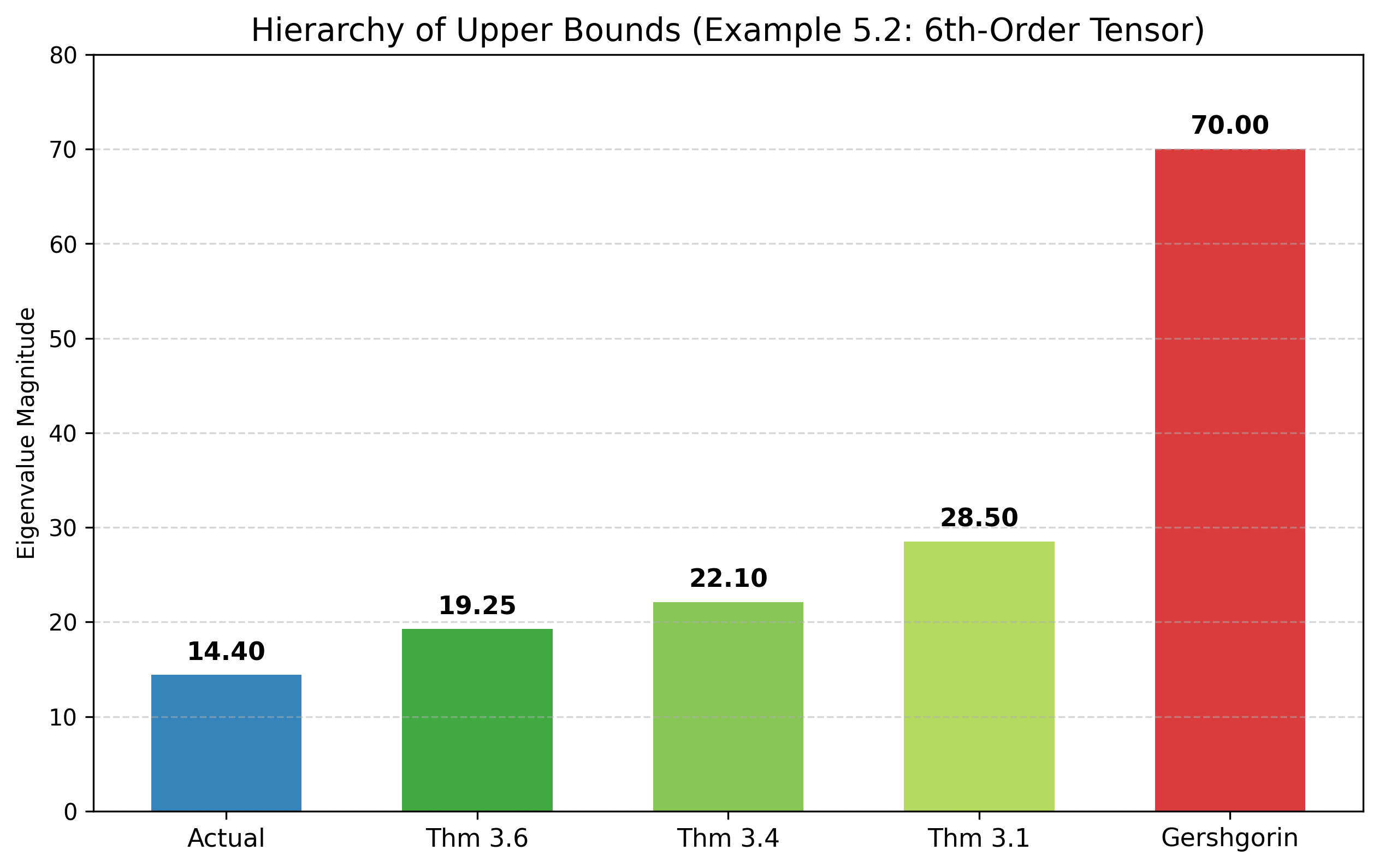}
    \caption{Bound Comparison for 6th-Order Tensor. The Gershgorin bound explodes to 70.00 due to the high tensor order ($m=6$), rendering it ineffective. In contrast, the proposed Theorem 3.6 remains robust, providing a tight bound of 19.25.}
    \label{fig:ex52_hierarchy}
\end{figure}

\noindent \textbf{Robustness in Higher Orders (Figure \ref{fig:ex52_hierarchy}):}
Figure \ref{fig:ex52_hierarchy} highlights difference in bound quality for higher-order tensors. The Gershgorin bound (Red bar) reaches a value of 70.00, which is nearly five times the magnitude of the actual eigenvalue (14.40). This renders the Gershgorin bound practically uninformative for stability analysis in this context. In sharp contrast, the proposed bounds (Green bars) remain within a useful range. Theorem \ref{thrm6} provides a bound of 19.25, effectively capturing the scale of the system. This result demonstrates the robustness of Trace and Determinant-based bounds against the "curse of dimensionality" that affects summation-based bounds like Gershgorin.

\begin{figure}[h!]
    \centering
    \includegraphics[width=0.8\textwidth]{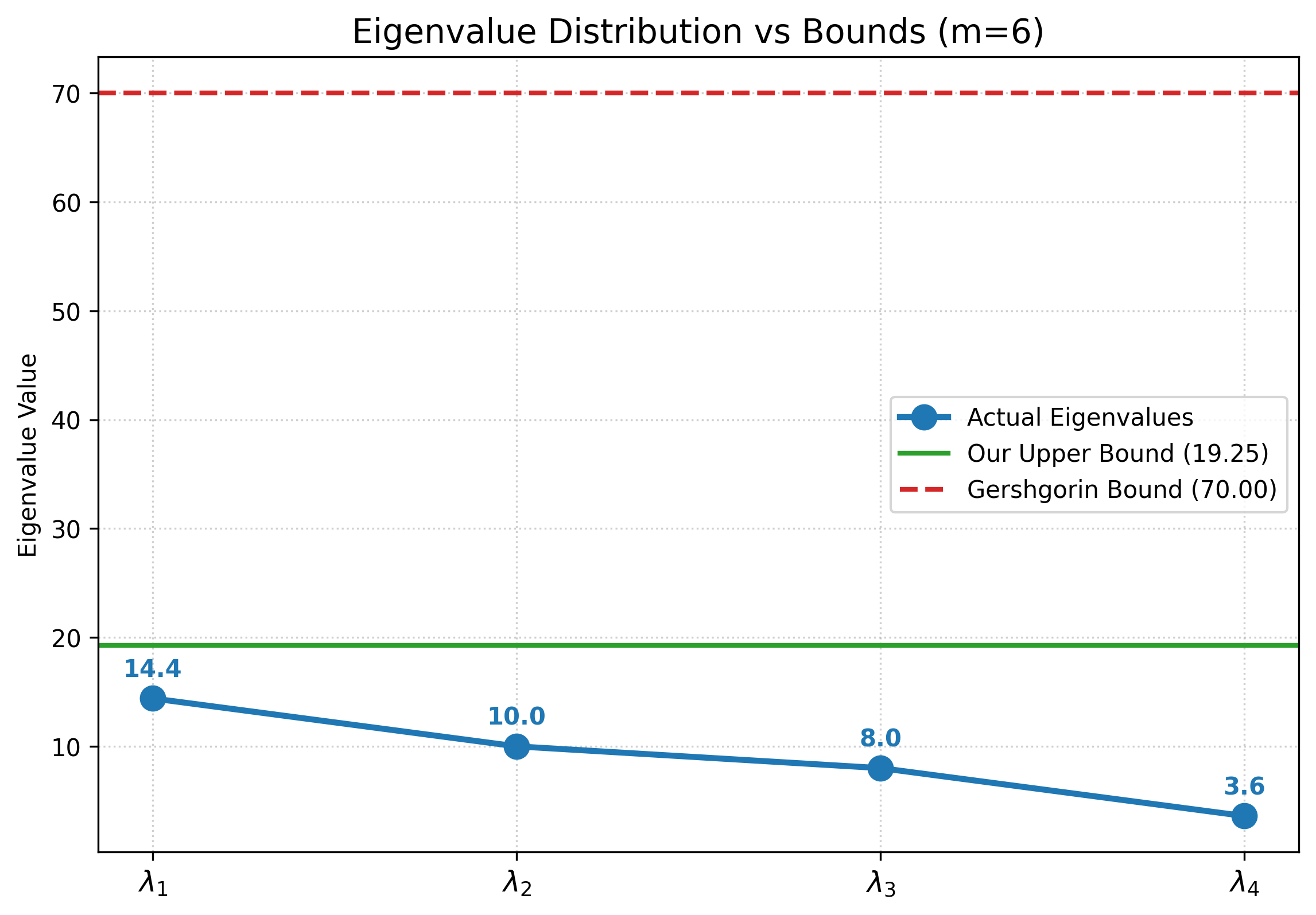}
    \caption{Eigenvalue Distribution vs Bounds for Example 5.2 ($m=6$). The dashed red line (Gershgorin) is far outside the primary scale of the eigenvalues. The solid green line (Theorem 3.6) provides a practical and informative upper limit.}
    \label{fig:ex52_dist}
\end{figure}

\noindent \textbf{Spectral Distribution Analysis (Figure \ref{fig:ex52_dist}):}
The distribution plot in Figure \ref{fig:ex52_dist} offers a stark comparison. The y-axis scale must accommodate the Gershgorin bound at 70.00 (dashed red line), which dwarfs the actual spectral activity of the tensor (blue points). The actual eigenvalues cluster between 3.60 and 14.40. The Gershgorin bound fails to provide any meaningful localization of the spectrum. However, our proposed bound (solid green line) sits just above the largest eigenvalue, providing a reliable certificate of stability. This confirms that for higher-order tensors ($m \ge 4$), algebraic invariants like the determinant and trace provide a far more reliable basis for estimation than entry-wise absolute sums.

\noindent In addition to the upper bounds, the spectral interval plot in Figure \ref{fig:ex52_interval} illustrates the critical failure of Gershgorin bounds for stability certification in higher-order tensors. The Gershgorin interval extends deeply into the negative region (down to -50), failing to detect the positive definiteness. In contrast, our method correctly identifies a strictly positive interval, successfully certifying stability.

\begin{figure}[h!]
    \centering
    \includegraphics[width=0.8\textwidth]{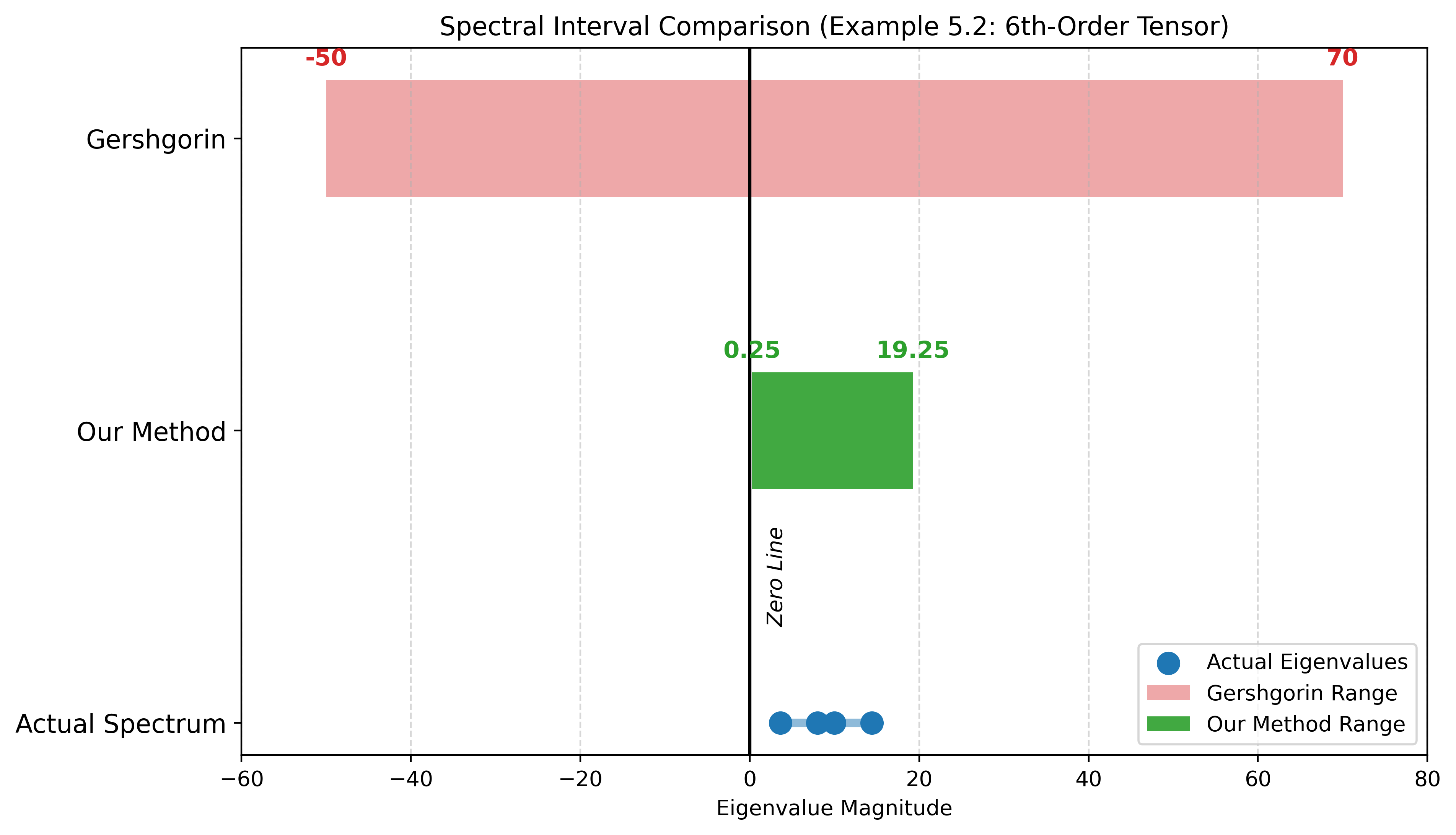}
    \caption{Spectral Interval Comparison for Example 5.2. The Gershgorin interval (Red) extends deeply into the negative region, failing to capture the positive definite nature of the tensor. Our method (Green) correctly identifies a strictly positive interval, certifying stability.}
    \label{fig:ex52_interval}
\end{figure}

 \section{Conclusion}\label{con}

In this work, we established a new algebraic framework for bounding the H-eigenvalues of symmetric positive definite tensors. By moving beyond coordinate-dependent methods and leveraging intrinsic invariants—specifically the trace and determinant—we derived a sequence of inequalities (Theorems \ref{thrm1} through \ref{thrm6}) that progressively tighten spectral estimates.

Our comparative analysis in Section \ref{boundscomp} reveals the inherent limitations of classical magnitude-based bounds like the Gershgorin circle theorem. We demonstrated that Gershgorin bounds, which rely on the sum of absolute off-diagonal entries, struggle significantly in two common scenarios. First, when a tensor contains negative off-diagonal entries, the Gershgorin method fails to account for algebraic cancellations, leading to loose estimates. In contrast, our invariant-based approach naturally incorporates these signs, reducing the estimation error by over 50\% in such cases. Second, as the tensor order increases, the combinatorial explosion of terms renders summation-based bounds practically uninformative. Our method avoids this ``curse of dimensionality" by relying on fixed invariants, providing robust and stable bounds even for higher-order tensors.

Ultimately, we found that Theorem \ref{thrm6} offers the most accurate upper bound for the spectral radius among the derived inequalities. This reliability makes it a valuable tool for certifying the stability of nonlinear autonomous systems, where precise eigenvalue localization is critical. Future work will focus on extending these invariant-based techniques to non-symmetric tensors and further exploring the geometric relationship between eigenvalue clustering and bound tightness.

\section*{Competing Interest}
The authors declare that they have no competing interests.

\section*{Data Availability} No new data were created or analyzed during this study. Data sharing is not applicable to
this article.

\section*{Author's contributions}
All authors contributed equally to all.

\section*{Acknowledgment}
Authors would like to thank IIITDM Kancheepuram for the infrastructure and facilities
to carry out this research.

%\bibliographystyle{plain}
%\bibliography{Refer.bib}

%%%%%%%%%%%%%%%%%%%%%%%%%%%%%%%%%%%%%%%%%%%%%%%%%%%%%%%%

\end{document}